\documentclass[12pt, reqno]{amsart}
\usepackage{amssymb,amsthm,amsmath}
\usepackage[numbers,sort&compress]{natbib}
\usepackage{amssymb,amsmath}
\usepackage{amsfonts}
\usepackage{mathrsfs}
\usepackage{latexsym}
\usepackage{amssymb}
\usepackage{amsthm}
\usepackage{indentfirst}
\date{\today}

\let\oldsection\section
\renewcommand\section{\setcounter{equation}{0}\oldsection}

\newtheorem{theorem}{Theorem}[section]
\newtheorem{lemma}{Lemma}[section]

\newtheorem{remark}{Remark}[section]

\def\ba{\begin{eqnarray}}
\def\ea{\end{eqnarray}}

\def\R{\Bbb R}

\newcommand{\beq}{\begin{equation}}
\newcommand{\eeq}{\end{equation}}
\newcommand{\ben}{\begin{eqnarray}}
\newcommand{\een}{\end{eqnarray}}
\newcommand{\beno}{\begin{eqnarray*}}
\newcommand{\eeno}{\end{eqnarray*}}

\allowdisplaybreaks
\begin{document}

\title[An attempt at axiomatization of extending mechanism]{An attempt at axiomatization of extending mechanism of solutions to the fluid dynamical systems}

\author{Jinkai~Li}
\address[Jinkai~Li]{South China Research Center for Applied Mathematics and Interdisciplinary Studies, School of Mathematical Sciences, South China Normal University, Zhong Shan Avenue West 55,
Tianhe District, Guangzhou 510631, China}
\email{jklimath@m.scnu.edu.cn; jklimath@gmail.com}

\author{Meng Wang}
\address[Meng Wang]{Department of Mathematics, Zhejiang University, Hangzhou
310027, China}
\email{mathdreamcn@zju.edu.cn}

\author{Wendong Wang}
\address[Wendong Wang]{School of Mathematical Sciences, Dalian University of Technology, Dalian, 116024,  China}
\email{wendong@dlut.edu.cn}
%

\keywords{interpolation inequalities of Besov type; Blow-up criteria; Navier-Stokes equations}
\subjclass[2010]{35Q30, 76D03}


\begin{abstract}
Note that some classic fluid dynamical systems such as the Navier-Stokes equations, Magnetohydrodynamics (MHD),
Boussinesq equations, and etc are observably different from each other but obey some energy inequalities of
the similar type. In this paper, we attempt to axiomatize the extending mechanism of solutions to these systems,
merely starting from several basic axiomatized conditions such as the local existence, joint property of
solutions and some energy inequalities. The results established have nothing to do with the concrete forms of the
systems and, thus, give the extending mechanisms in a unified way to all systems obeying the
axiomatized conditions. The key tools are several new multiplicative interpolation inequalities of Besov type,
which have their own interests.
\end{abstract}

\maketitle

\allowdisplaybreaks

\section{Introduction}\label{sec1}
The fluid dynamical models, just mention some of them, include
the Navier-Stokes equations, Magnetohydrodynamics (MHD) and the liquid crystals. Generally speaking, for those models with viscosities, the local well-posedness can
be established in somewhat standard way; however, recalling that it is a well-known open question to prove the global existence of smooth solutions to the
three-dimensional Navier-Stokes equations, though a lot
of attentions have been made, we are still far from the complete
mathematical understanding of the the global well-posedness of
these systems.

To understand the possible singularity or global regularity, it may
be helpful for us to study some blow-up criteria. The aim of establishing the blow-up criteria is to find such conditions as weak
as possible that ensure the global regularity of the solutions.
There have been a
lot of works on establishing the blow-up criteria for the classic fluid dynamical
systems, in particular for the systems mentioned above, in the existing
literatures. There are two well-known blow-up criteria
for the Navier-Stokes equations, the Ladyzhenskaya--Prodi--Serrin
type (see, e.g., \cite{Serrin,Struwe}) and the Beale--Kato--Majda type (see \cite{BKM}). The
Ladyzhenskaya--Prodi--Serrin blow-up criteria imply that if  the following Serrin condition
\begin{eqnarray*}
u\in L^q(0,T; L^p(\mathbb R^3))\quad \textrm{with} \quad \frac 2 q+\frac 3p\le 1, \quad 3< p\leq \infty,
\end{eqnarray*}
holds for a positive time $T$ and a pair of $(p,q)$, then the solution to the
Navier-Stokes equations will not blow-up at time $T$; while
the Beale--Kato--Majda blow up criteria tell us that as long as
the $L^1(0,T; BMO(\mathbb R^3))$ norm of the vorticity is finite, then the
solution to the Navier-Stokes equations can be extended beyond the
time $T$. Both cases can be proved by the energy method with Gagliardo-Nirenberg interpolation inequality or some embedding inequality of logarithmic type. Note that the endpoint case of  the Serrin condition $u\in L^\infty(0,T; L^3(\R^3))$ is different, which was proved
by Escauriaza,  Seregin and  \v{S}ver\'{a}k \cite{ESS} by employing blow-up analysis and the backward uniqueness property of the parabolic operator.

Many generalizations and extensions to other systems of these two kinds of blow-up criteria have been made. For the Navier-Stokes and MHD equations, Ladyzhenskaya--Prodi--Serrin's  and Beale--Kato--Majda's criteria were considered in Sobolev space or Besov space,  for example see \cite{CMZ1,CMZ2,FangQian,Giga, Kozono, HeWang, HeXin2005, Wu1, Wu2, Zhou} and the references therein.

We note that though the systems from the fluid dynamics may be observably different from each other, they may obey some energy inequalities of the same type. For example, one can easily check that the Navier-Stokes equations, MHD and Boussineq equations satisfy the same type energy inequalities stated in (H4), (H4') and (H5), below. Therefore, there should be some common features for these systems, such as the well-posedness under the same assumptions on the initial data, and the same type blow-up criteria under the same conditions on the solutions. Keeping this in mind, it will be very interesting to reveal all these common features determined by the same characteristics of the systems from the fluid dynamics, while this paper is employed as the first attempt to exploit a tip of the iceberg of these common features.

Specifically, we list some axiomatized conditions, by which we get a unified proof for blow-up criteria of Ladyzhenskaya--Prodi--Serrin type and Beale--Kato--Majda type in Besov space of the systems satisfying these conditions.
Without such restrictions, these systems may have some blow-up solutions. Recently Tao \cite{Tao} proved that in the averaged Navier-Stokes equations, it is assumed that the nonlinear term satisfies all the classical harmonic analysis estimates and the
fundamental cancellation property, so as to construct the exploding solution.
Throughout this paper, we use $(\mathscr P)$ to denote the Cauchy problem of an arbitrary PDE system in $\mathbb R^d$.
The problem $(\mathscr P)$ considered in this paper is supposed to satisfies (H1)--(H3) and at least one of (H4) and (H4'):
\begin{itemize}
  \item [(H1)] \textbf{(Local existence)} For any $u_0\in H^1(\mathbb R^d)$ (if some other conditions, such as $\text{div}\,u=0$, is included in $(\mathscr P)$, then $u_0$ is required to satisfy the same conditions), there exists a unique solution $u$ to $(\mathscr P)$ on $\mathbb R^d\times(0,T_*)$, with initial data $u_0$, and for any $T<T_*$, there hold
      $$
       u\in C([0,T]; H^1(\mathbb R^d))\cap L^2(0,T; H^2(\mathbb R^d)),\quad \partial_tu\in L^2(0,T; L^2(\mathbb R^d)),
      $$
     where $T_*$ is a positive number depending only on the upper bound of the initial norm $\|u_0\|_{H^1(\mathbb R^d)}$.
  \item [(H2)] \textbf{(Joint property)} For any two solutions $u_1$ and $u_2$ to $(\mathscr P)$ on $\mathbb R^d\times(0,T_1)$ and on $\mathbb R^d\times(T_1,T_2)$, respectively, with initial data $u_0$ and $u_1(T_1)$, respectively, such that $u_1\in C([0,T_1]; H^1(\mathbb R^d))$ and $u_2\in C([T_1,T_2); H^1(\mathbb R^d))$, then the joint function $u$ defined as
      $$
      u:=\left\{
      \begin{array}{l}
      u_1(x,t),\quad t\in[0,T_1),\\
      u_2(x,t),\quad t\in[T_1, T_2),
      \end{array}
      \right.
      $$
      is a solution to $(\mathscr P)$ on $\mathbb R^d\times(0,T_2)$, with initial data $u_0$.
      \item [(H3)] \textbf{(Basic energy)} For any solution $u$ to $(\mathscr P)$ on $\mathbb R^d\times(0,T)$, it satisfies
      $$
      \sup_{0\leq s\leq t}\|u\|_{L^2(\mathbb R^d)}\leq \mathcal M(t),
      $$
      for any $t\in(0,T)$, where $\mathcal M$ is a continuous nondecreaing function on $[0,\infty)$, determined by the initial data.
  \item [(H4)] \textbf{(First energy inequality)} For any solution $u$ to $(\mathscr P)$ on $\mathbb R^d\times(0,T)$, it satisfies the energy inequality
      $$
      \frac{d}{dt}\|\nabla u\|_{L^2(\mathbb R^d)}^2+c_1\|\Delta u\|_{L^2(\mathbb R^d)}^2\leq c_1'\int_{\mathbb R^d}|u|^2|\nabla u|^2 dx,
      $$
      for $t\in(0,T)$, where $c_1$ and $c_1'$ are two positive constants.
  \item [(H4')] \textbf{(First energy inequality')} For any solution $u$ to $(\mathscr P)$ on $\mathbb R^d\times(0,T)$, it satisfies the energy inequality
      $$
      \frac{d}{dt}\|\nabla u\|_{L^2(\mathbb R^d)}^2+c_2\|\Delta u\|_{L^2(\mathbb R^d)}^2\leq c_2'\|\nabla u\|_{L^3(\mathbb R^d)}^3,
      $$
      for $t\in(0,T)$, where $c_2$ and $c_2'$ are two positive constants.
\end{itemize}
The following hypothesis on the second energy may also be used in some specific case (the case that $s=1$ in Theorem \ref{thm2}, below):
\begin{itemize}
  \item [(H5)] \textbf{(Second energy inequality)}
  For any solution $u$ to $(\mathscr P)$ on $\mathbb R^d\times(0,T)$, it satisfies the energy inequality
      $$
      \frac{d}{dt}\|\Delta u\|_{L^2(\mathbb R^d)}^2+c_3\|\nabla\Delta u\|_{L^2(\mathbb R^d)}^2\leq c_3'\int_{\mathbb R^d}(|\nabla u||\nabla^2u|^2+|\nabla u|^4) dx,
      $$
      for $t\in(0,T)$, where $c_3$ and $c_3'$ are two positive constants.
\end{itemize}

\begin{remark}
It can be verified that both the Navier-Stokes system and the MHD system meet the above conditions (H1)-(H4),  (H4') and (H5).
\end{remark}

The first result of this paper is the following
\begin{theorem}
  \label{thm1}
  Given a positive time $T_*$, and let $(\mathscr P)$ be the Cauchy
  problem of an arbitrary PDE system, such that it satisfies the hypothesises (H1)--(H4). Let $u$ be a solution to $(\mathscr P)$ on $\mathbb R^d\times(0,T_*)$, satisfying
  \begin{eqnarray*}
  u\in L^q(0,T_*; B^{-s}_{p,\infty}(\mathbb R^d)),
  \end{eqnarray*}
  for some constants $s, p$ and $q$, such that
  $$
  \frac2q+\frac dp=1-s, \quad \mbox{with } p\in\left(\frac{d}{1-s},\infty\right]\mbox{and } s\in(0,1).
  $$
  Then the solution $u$ can be extended uniquely beyond $T_*$.
\end{theorem}

  Note that the hypothesis (H4') is stronger than (H4). Therefore, one can expect that if the system $(\mathscr P)$ satisfies the stronger hypothesis (H4') instead of (H4), then a better result than Theorem \ref{thm1} should hold. In fact, we have the following

\begin{theorem}
  \label{thm2}
  Given a positive time $T_*$, and let $(\mathscr P)$ be the Cauchy
  problem of an arbitrary PDE system, such that it satisfies the hypothesises (H1)--(H3) and (H4'). Let $u$ be a solution to $(\mathscr P)$ on $\mathbb R^d\times(0,T_*)$, satisfying
  \begin{eqnarray*}
  u\in L^q(0,T_*; B^{s}_{p,\infty}(\mathbb R^d)),
  \end{eqnarray*}
  for some constants $s, p$ and $q$, such that
  $$
  \frac2q+\frac dp=1+s, \quad \mbox{with } p\in\left(\frac{d}{1+s},\infty\right]\mbox{and } s\in(-1,1].
  $$
  Then, for the case that $(p,s)\not=(\infty,1)$, the solution $u$ can be extended uniquely beyond $T_*$. While for the case that $(p,s)=(\infty,1)$, the solution $u$ can also be extended uniquely beyond $T_*$, if we have further that (H5) holds, and $d=2,3$.
\end{theorem}

\begin{remark}
(1) It should be noticed that if considering the Cauchy problem to the MHD equations in three dimensions, Theorem 2 has
already been obtained in \cite{CMZ2}; however, comparing with their work, one of our biggest features here is that we
need neither the concrete structure of the equations nor the Bony decomposition and commutator estimates.

(2) Since $L^p(\R^d)\varsubsetneq B^{0}_{p,\infty}(\R^d)$ for $s=0$, the classic Ladyzhenskaya--Prodi--Serrin criterion
is included in Theorem \ref{thm2}. Theorem \ref{thm1}-\ref{thm2} still hold if replacing the
inhomogeneous Besov spaces by the corresponding
homogeneous ones. In fact, by checking the proof, it suffices to verify that
Lemma \ref{lemL1} and Lemma \ref{lemL3} continue to hold if replacing the inhomogeneous Besov spaces there by
the homogeneous ones. Thanks to this and noticing that $\dot{B}_{p,\infty}^{s}(\R^d)\subset
B_{p,\infty}^{s}(\R^d) $ for $s<0$ and $\dot{B}_{p,\infty}^s(\R^d)\cap
L^p(\R^d)=B_{p,\infty}^{s}(\R^d)$ for $s> 0$ , the condition for $u$ in Theorem \ref{thm2} can be relaxed to
  \begin{eqnarray*}
  \left\{
    \begin{array}{lr}
      u\in L^q(0,T_*; B^{s}_{p,\infty}(\mathbb R^d)), & -1<s\leq0, \\
      u\in L^q(0,T_*; \dot{B}^{s}_{p,\infty}(\mathbb R^d)), & 0<s\leq1,
    \end{array}
  \right.
  \end{eqnarray*}
with $p, q,$ and $s$ obeying the relations stated in Theorem \ref{thm2}. In particular, in the case that $(p,q,s)=(\infty, 1, 1)$, the relaxed condition reduces to $\nabla u\in L^{1}(0,T_*; \dot{B}^{0}_{\infty,\infty}(\mathbb R^d))$, which is of the well-known Beale--Kato--Majda type criteria.

(3) Note that the endpoint case $(p,q,s)=(\infty,\infty,-1)$, i.e.\,the case that $u\in L^\infty(0,T; B^{-1}_{\infty,\infty})$, is excluded in both Theorem \ref{thm1} and Theorem \ref{thm2}. In fact, through the regularity
of the Leray-Hopf weak solutions to the three dimensional Navier-Stoke equation was obtained under the condition that $u\in C((0,T]; B^{-1}_{\infty,\infty})$, see \cite{ches}, it is still unknown if it can be relaxed to $u\in L^\infty(0,T; B^{-1}_{\infty,\infty})$.
\end{remark}

The rest of the paper is organized as follows. In section 2, we introduce the definition of Besov space and several new interpolation inequalities of Besov type, which are our main technical tools.
In section 3, we prove Theorem \ref{thm1} and Theorem \ref{thm2} under the abstract assumptions (H1)--(H5).

\section{Preliminaries}

\subsection{Littlewood-Paley decomposition}

 Let us recall some basic facts on Littlewood-Paley theory(see \cite{Che}  for more details). Choose two nonnegative radial functions $\chi,\phi\in { \mathcal{S}}(R^n)$ supported respectively in
$\{\xi\in R^n,|\xi|\le \frac{4}{3}\}$ and $\{\xi\in R^n, \frac{3}{4}\le |\xi|\le \frac{8}{3}\}$ such that for any $\xi\in R^n$,
$$
\chi(\xi)+\sum_{j\ge 0}\phi(2^{-j}\xi)=1.
$$
The frequency localization operator $\Delta_j$ and $S_j$ are defined by
\begin{eqnarray*}
&&\Delta_j f=\phi(2^{-j}D)f=2^{nj}\int_{R^n}h(2^{j}y)f(x-y)dy,~~~~\mbox{for}~~ j\ge 0,\\
&&S_j f=\chi(2^{-j}D)f=\sum_{-1\le k\le j-1}\Delta_{k}f=2^{nj}\int_{R^n}\tilde{h}(2^{j}y)f(x-y)dy,\\
&&\Delta_{-1}f=S_0f, ~~\Delta_j f=0~~ \mbox{for}~~ j\le -2,
\end{eqnarray*}
where $h={ \mathcal{F}}^{-1}\phi$, $\tilde{h}={ \mathcal{F}}^{-1}\chi$.  With this choice of $\phi$, it is easy to verify that
\begin{eqnarray}\label{eq:A.1}
\Delta_j\Delta_kf=0,~~ {\rm if}~~|j-k|\geq 2;~~\Delta_j(S_{k-1}{g}\Delta_kf)=0,~~ {\rm if}~~|j-k|\geq 5.
\end{eqnarray}

In terms of $\Delta_j$, the norm of the inhomogeneous Besov space $B^s_{p,q}$ for $s\in R,$ and $p,q\geq 1$ is defined by
\begin{eqnarray*}
\|f\|_{B^s_{p,q}}\doteq\big\|\{2^{js}\|\Delta_jf\|_p\}_{j\ge -1}\big\|_{\ell^q},
\end{eqnarray*}
and
\begin{eqnarray*}
\|f\|_{B^s_{p,\infty}}\doteq\sup_{j\ge -1}\{2^{js}\|\Delta_jf\|_p\},
\end{eqnarray*}

We will constantly use the following Bernstein's inequality \cite{Che}.
\begin{lemma}[Bernstein inequality]\label{lem:Berstein}
Let $c\in (0,1)$ and $R>0$. Assume that $1\leq p\leq q\leq \infty$ and $f\in L^p(R^n)$. Then
\begin{eqnarray*}
&&{\rm supp} \hat{f}\subset\big\{|\xi|\leq R\big\}\Rightarrow \|\partial^{\alpha}f\|_{q}\leq CR^{|\alpha|+n(\frac1p-\frac1q)}\|f\|_{p},\\
&&{\rm supp} \hat{f}\subset\big\{cR \leq |\xi|\leq R\big\}\Rightarrow \|f\|_{p}\leq CR^{-|\alpha|}\sup_{|\beta|=|\alpha|}\|\partial^{\beta}f\|_{p},
\end{eqnarray*}
where the constant $C$ is independent of $f$ and $R$.
\end{lemma}

\subsection{New interpolation inequalities of Besov type}

\begin{lemma}\label{lemL1}
  Let $s\in(0,\infty)$. Then, for any nonzero function $f\in H^2\cap B_{\infty,\infty}^{-s}$, the following interpolation inequality holds
  $$
    \|f\|_p\leq C_{n,p,s}\|f\|_{B^{-s}_{\infty,\infty}}^{1-\frac2p}\|f\|_{H^1} ^{\frac4p-s(1-\frac2p)}\|f\|_{H^2}^{s(1-\frac2p)-\frac2p},
  $$
  for any $p\in\left(2+\frac2s,2+\frac4s\right)$,
  where
  \begin{eqnarray*}
  &C_{n,p,s}=C_n\max\left\{\frac{1}{2^{s(1-\frac2p)-\frac2p}-1},\frac{1} {1-2^{s(1-\frac2p)-\frac4p}}\right\},
  \end{eqnarray*}
  for a positive constant $C_n$ depending only on $n$.
\end{lemma}

\begin{proof}
For any integer $k\geq0$ and $p\in[2,\infty]$,
it follows from the H\"older and Bernstein inequalities that
\begin{align*}
  \|\Delta_kf\|_p\leq&\|\Delta_kf\|_2^{\frac2p}\|\Delta_kf\|_\infty^{1-\frac2p} \leq C_n2^{-\frac{2k}{p}}\|\nabla\Delta_kf\|_2^{\frac2p}\left(2^{ks} 2^{-ks}\|\Delta_kf\|_\infty\right)^{1-\frac2p}\\
  \leq&C_n2^{k\left[s(1-\frac2p)-\frac2p\right]}\|\nabla\Delta_kf\|_2^{\frac2p} \|f\|_{B^{-s}_{\infty,\infty}}^{1-\frac2p},
\end{align*}
from which, using the Bernstein inequality again, one arrives at
\begin{align*}
  \|\Delta_kf\|_p\leq&C_n2^{k[s(1-\frac2p)-\frac4p]}\|\Delta\Delta_kf\|_2^{\frac2p} \|f\|_{B^{-s}_{\infty,\infty}}^{1-\frac2p},
\end{align*}
for any integer $k\geq0$. Thus, noticing that $\|\Delta_kg\|_2\leq C_n\|g\|_2,$
we obtain
\begin{eqnarray}
  &&\|\Delta_kf\|_p\leq C_n2^{k\left[s(1-\frac2p)-\frac2p\right]}\|\nabla f\|_2^{\frac2p} \|f\|_{B^{-s}_{\infty,\infty}}^{1-\frac2p},\label{L2}\\
  &&\|\Delta_kf\|_p\leq C_n 2^{-k[\frac4p -s(1-\frac2p)]}\|\Delta f\|_2^{\frac2p} \|f\|_{B^{-s}_{\infty,\infty}}^{1-\frac2p},\label{L3}
\end{eqnarray}
for any integer $k\geq0$ and $p\in[2,\infty]$. For $k=-1$, by the H\"older inequality, and noticing that $\|\Delta_{-1}f\|_2\leq C_n\|f\|_2$, one deduces \begin{align}
  \|\Delta_{-1}f\|_p\leq &\|\Delta_{-1}f\|_2^{\frac2p}\|\Delta_{-1}f \|_\infty^{1-\frac2p}\leq C_n2^{-s(1-\frac2p)}\|f\|_2^{\frac2p} \|f\|_{B_{\infty,\infty}^{-s}}^{1-\frac2p}\nonumber\\
  \leq&C_n2^{-[s(1-\frac2p)-\frac2p]}\|f\|_2^{\frac2p} \|f\|_{B_{\infty,\infty}^{-s}}^{1-\frac2p},\label{L4}
\end{align}
for any $p\in[2,\infty]$.

For any $p\in(2+\frac2s,2+\frac4s)$, one can easily verify
\begin{equation}\label{L1}
s\left(1-\frac2p\right)-\frac2p>0\quad\mbox{and}\quad \frac4p -s\left(1-\frac2p\right)>0.
\end{equation}
On account of this, for any integer $k_0\geq0$, it follows from (\ref{L2})--(\ref{L4}) that
\begin{align}
  \|f\|_p=&\left\|\sum_{k=-1}^\infty\Delta_kf\right\|_p\leq\|\Delta_{-1} f\|_p+\sum_{k=0}^{k_0}\|\Delta_kf\|_p+\sum_{k=k_0+1}^\infty\|\Delta_kf\|_p \nonumber\\
  \leq&C_n2^{-[s(1-\frac2p)-\frac2p]}\|f\|_2^{\frac2p}\|f\|_{B_{\infty,\infty}^{-s}}^{1-\frac2p} +C_n\sum_{k=0}^{k_0}2^{k[s(1-\frac2p)-\frac2p]}\|\nabla f\|_2^{\frac2p} \|f\|_{B_{\infty,\infty}^{-s}}^{1-\frac2p}\nonumber\\
  &+C_n\sum_{k=k_0+1}^{\infty}2^{-k[\frac4p-s(1-\frac2p)]}\|\Delta f\|_2^{\frac2p} \|f\|_{B_{\infty,\infty}^{-s}}^{1-\frac2p}\nonumber\\
  \leq&C_n\|f\|_{B_{\infty,\infty}^{-s}}^{1-\frac2p}\left(
  \sum_{k=-1}^{k_0}2^{k[s(1-\frac2p)-\frac2p]}\|f\|_{H^1}^{\frac2p}+ \sum_{k=k_0+1}^{\infty}2^{-k[\frac4p-s(1-\frac2p)]}\|\Delta f\|_2^{\frac2p}\right)\nonumber\\
  \leq&C_{n,p,s}\|f\|_{B_{\infty,\infty}^{-s}}^{1-\frac2p}\left(
  2^{k_0[s(1-\frac2p)-\frac2p]}\|f\|_{H^1}^{\frac2p}+ 2^{-(k_0+1)[\frac4p-s(1-\frac2p)]}\|\Delta f\|_2^{\frac2p}\right),\label{L5}
\end{align}
for any $p\in(2+\frac2s, 2+\frac4s)$.

Note that $\frac{\|\Delta f\|_2+\|f\|_{H^1}}{\|f\|_{H^1}}\geq1$, there is a unique integer $k_0\geq0$, such that
$$
2^{k_0}\leq\frac{\|\Delta f\|_2+\|f\|_{H^1}}{\|f\|_{H^1}}<2^{k_0+1}.
$$
Choosing such $k_0$ in (\ref{L5}), and recalling (\ref{L1}), we obtain
\begin{align*}
  \|f\|_p\leq&C_{n,p,s}\|f\|_{B_{\infty,\infty}^{-s}}^{1-\frac2p} \left(
  \frac{\|\Delta f\|_2+\|f\|_{H^1}}{\|f\|_{H^1}}\right)^{s(1-\frac2p)-\frac2p} \|f\|_{H^1}^{\frac2p}\\
  &+C_{n,p,s}\|f\|_{B_{\infty,\infty}^{-s}}^{1-\frac2p} \left(\frac{\|f\|_{H^1}}{\|\Delta f\|_2+\|f\|_{H^1}}\right)^{\frac4p-s(1-\frac2p)}\|\Delta f\|_2^{\frac2p}\\
  \leq&C_{n,p,s}\|f\|_{B_{\infty,\infty}^{-s}}^{1-\frac2p} \|f\|_{H^1}^{\frac4p-s(1-\frac2p)}\|f\|_{H^2}^{s(1-\frac2p)-\frac2p},
\end{align*}
for any $p\in(2+\frac2s, 2+\frac4s)$. This completes the proof.
\end{proof}

\begin{lemma}
  \label{lemL3}
  Let $s\in(-\infty,1)$. Then, for any nonzero function $f\in B_{\infty,\infty}^{s}$ such that $\nabla f\in H^1$, we have
  \begin{equation*}
    \|\nabla f\|_q\leq C_{n,q,s}\|f\|_{B_{\infty,\infty}^s}^{1-\frac2q}\|\nabla f\|_2^{\frac2q-(1-s)(1-\frac2q)}\|\nabla f\|_{H^1}^{(1-s)(1-\frac2q)},
  \end{equation*}
  for any $q\in(2,2+\frac{2}{1-s})$, where
  $$
  C_{n,q,s}=C_n\max\left\{\frac{1}{2^{(1-s)(1-\frac2q)}-1}, \frac{1}{1-2^{-[\frac2q-(1-s)(1-\frac2q)]}}\right\},
  $$
  for a positive constant $C_n$ depending only on $n$.
\end{lemma}

\begin{proof}
  For any $q\in[2,\infty]$ and any integer $k\geq0$,
  it follows from the H\"older and Bernstein inequalities that
  $$
  \|\nabla\Delta_kf\|_q\leq\|\nabla\Delta_kf\|_2^{\frac2q}\|\nabla\Delta_kf \|_\infty^{1-\frac2q}\leq C_n2^{k(1-s)(1-\frac2q)}\|\nabla\Delta_kf\|_2^{\frac2q}\|f\|_{B^s_{\infty, \infty}}^{1-\frac2q},
  $$
  from which, using the Bernstein inequality again, we have
  $$
  \|\nabla\Delta_kf\|_q\leq C_n2^{-k[\frac2q-(1-s)(1-\frac2q)]}\|\Delta\Delta_k f\|_2^{\frac2q}\|f\|_{B^s_{\infty,\infty}}^{1-\frac2q}.
  $$
  Thanks to the above two inequalities, noticing that $\|\Delta_kg\|_2\leq C_n\|g\|_2$, we get
  \begin{eqnarray}
    &&\|\nabla\Delta_kf\|_q\leq C_n2^{k(1-s)(1-\frac2q)}\|\nabla f\|_2^{\frac2q}\|f\|_{B^s_{\infty,\infty}}^{1-\frac2q},\label{L8}\\
    &&\|\nabla\Delta_kf\|_q\leq C_n2^{-k[\frac2q-(1-s)(1-\frac2q)]}\|\Delta f\|_2^{\frac2q}\|f\|_{B^s_{\infty,\infty}}^{1-\frac2q},\label{L9}
  \end{eqnarray}
  for any $q\in[2,\infty]$ and integer $k\geq0$. For $k=-1$, by the H\"older and Bernstein inequalities, we have
  \begin{align}
    \|\nabla\Delta_{-1}f\|_q\leq&\|\nabla\Delta_{-1}f\|_2^{\frac2q}\|\nabla\Delta_{-1} f\|_\infty^{1-\frac2q}\leq C_n2^{s(1-\frac2q)}\|\nabla f\|_2^{\frac2q}\|f\|_{B^s_{\infty,\infty}}^{1-\frac2q}\nonumber\\
    \leq&C_n2^{-(1-s)(1-\frac2q)}\|\nabla f\|_2^{\frac2q}\|f\|_{B^s_{\infty,\infty}}^{1-\frac2q}, \label{L10}
  \end{align}
  for any $q\in[2,\infty]$.

  For any $q\in(2,2+\frac{2}{1-s})$, one can easily check that
  $$
  (1-s)\left(1-\frac2q\right)>0\quad\mbox{and}\quad\frac2q-(1-s) \left(1-\frac2q\right)>0.
  $$
  Thus, it follows from (\ref{L8})--(\ref{L10}) that
  \begin{align}
    \|\nabla f\|_q=&\left\|\sum_{k=-1}^\infty\nabla\Delta_kf\right\|_q\leq \sum_{k=-1}^{k_0}\|\nabla\Delta_kf\|_q+\sum_{k=k_0+1}^\infty\|\nabla\Delta_k f\|_q\nonumber\\
    \leq&C_n\|f\|_{B^s_{\infty,\infty}}^{1-\frac2q}\left(\sum_{k=-1}^{k_0} 2^{k(1-s)(1-\frac2q)}\|\nabla f\|_2^{\frac2q}+\sum_{k=k_0+1}^\infty 2^{-k[\frac2q-(1-s)(1-\frac2q)]}\|\Delta f\|_2^{\frac2q}\right)\nonumber\\
    \leq&C_{n,p,s}\|f\|_{B^s_{\infty,\infty}}^{1-\frac2q}\left( 2^{k_0(1-s)(1-\frac2q)}\|\nabla f\|_2^{\frac2q}+ 2^{-(k_0+1)[\frac2q-(1-s)(1-\frac2q)]}\|\Delta f\|_2^{\frac2q}\right),\label{L11}
  \end{align}
  for any $k_0\geq0$. Take $k_0\geq0$ be the unique integer such that
  $$
  2^{k_0}\leq\frac{\|\Delta f\|_2+\|\nabla f\|_2}{\|\nabla f\|_2}<2^{k_0+1},
  $$
  then it follows from (\ref{L11}) that
  \begin{align*}
    \|\nabla f\|_q\leq&C_{n,p,s}\|f\|_{B^s_{\infty,\infty}}^{1-\frac2q}\left( \frac{\|\Delta f\|_2+\|\nabla f\|_2}{\|\nabla f\|_2}\right)^{(1-s)(1-\frac2q)}\|\nabla f\|_2^{\frac2q}\\
    &+C_{n,p,s}\|f\|_{B^s_{\infty,\infty}}^{1-\frac2q}\left( \frac{\|\nabla f\|_2}{\|\Delta f\|_2+\|\nabla f\|_2}\right)^{\frac2q-(1-s)(1-\frac2q)}\|\Delta f\|_2^{\frac2q}\\
    \leq&C_{n,p,s}\|f\|_{B^s_{\infty,\infty}}^{1-\frac2q} \|\nabla f\|_2^{\frac2q-(1-s)(1-\frac2q)}\|\nabla f\|_{H^1}^{(1-s)(1-\frac2q)},
  \end{align*}
  for any $q\in(2,2+\frac{2}{1-s})$, proving the conclusion.
\end{proof}

We also will use an integral in time version of the logarithmic type inequality stated in the next lemma. Some similar inequalities can be found in Huang--Wang \cite{HW} and Hong--Li--Xin \cite{HLX}

\begin{lemma}
  \label{lemL5}
  Suppose that $n=2,3$. Then the following two inequalities hold
  \begin{align*}
    \int_{t_1}^{t_2}\|\nabla f\|_\infty dt\leq& C\left[\left(\int_{t_1}^{t_2}\|f\|_{B_{\infty,\infty}^1} dt\right)\log\left(\int_{t_1}^{t_2}\|\nabla\Delta f\|_2 dt+e\right) +1\right],
  \end{align*}
  for any $f$ such that the quantities in the formulas make sense and are finite, where $C$ is a positive constant independent of $t_1$ and $t_2$.
\end{lemma}

\begin{proof}
  We first prove the first inequality. By the definition of $B^1_{\infty,\infty}$, and the Bernstein inequality, it is obvious that
  \begin{equation}\label{L15}
  \|\nabla\Delta_kf\|_\infty\leq C2^k\|\Delta_kf\|_\infty\leq C\|f\|_{B^1_{\infty,\infty}},
  \end{equation}
  for any integer $k\geq-1$. Using the Bernstein inequality again, and noticing that $\|\Delta_kg\|_2\leq C_n\|g\|_2$, we deduce
  \begin{align}
    \|\nabla\Delta_kf\|_\infty\leq C2^{(\frac n2-2)k}\|\nabla\Delta\Delta_kf\|_2\leq C2^{-k(2-\frac n2)}\|\nabla\Delta f\|_2, \label{L16}
  \end{align}
  for any integer $k\geq0$. With the aid of the above two inequalities, we have
  \begin{align*}
    \|\nabla f\|_\infty =&\left\|\sum_{k=-1}^\infty\nabla\Delta_kf\right\|_\infty \leq\sum_{k=-1}^{k_0}\|\nabla\Delta_kf\|_\infty+\sum_{k=k_0+1}^\infty \|\nabla\Delta_kf\|_\infty\\
    \leq&C(k_0+1)\|f\|_{B^1_{\infty,\infty}}+C\sum_{k=k_0+1}^\infty 2^{-k(2-\frac n2)}\|\nabla\Delta f\|_2\\
    \leq&C\left[(k_0+1)\|f\|_{B^1_{\infty,\infty}}+2^{-(k_0+1)(2-\frac n2)}\|\nabla\Delta f\|_2\right],
  \end{align*}
  for any integer $k_0\geq0$. Integrating the above inequality with respect to $t$ over the interval $(t_1,t_2)$, and choosing $k_0\geq0$ be the unique integer such that
  $$
  k_0\leq\left[\left(2-\frac n2\right)\log 2\right]^{-1}\log\left(\int_{t_1}^{t_2}\|\nabla \Delta f\|_2+1\right)<k_0+1,
  $$
  then we obtain
  \begin{align*}
    \int_{t_1}^{t_2}\|\nabla f\|_\infty dt\leq&C(k_0+1)\int_{t_1}^{t_2}\|f\|_{B^1_{\infty, \infty}}dt+C2^{-(k_0+1)(2-\frac n2)}\int_{t_1}^{t_2}\|\nabla\Delta f\|_2dt\\
    \leq&C\left(\int_{t_1}^{t_2}\|f\|_{B^1_{\infty, \infty}}dt\right)\left[ \log \left(\int_{t_1}^{t_2}\|\nabla \Delta f\|_2+1\right)+1\right]\\
    &+Ce^{-[(2-\frac n2)\log 2](k_0+1)}\int_{t_1}^{t_2}\|\nabla \Delta f\|_2dt\\
    \leq&C\left[\left(\int_{t_1}^{t_2}\|f\|_{B^1_{\infty, \infty}}dt\right) \log \left(\int_{t_1}^{t_2}\|\nabla \Delta f\|_2+e\right)+1\right],
  \end{align*}
  proving the fist inequality.
\end{proof}

\section{Proofs of Theorem \ref{thm1} and Theorem \ref{thm2}}

In this section, we give the proof of Theorem \ref{thm1} and Theorem \ref{thm2}.

\begin{proof}[\textbf{Proof of Theorem \ref{thm1}}]
   Note that $B^{-s}_{p,\infty}\hookrightarrow B^{-s-\frac np}_{\infty,\infty}$, it follows that
   $$
   L^q(0,T; B^{-s}_{p,\infty})\hookrightarrow L^q(0,T;B^{-s_1}_{\infty,\infty}),\quad s_1:=s+\frac np.
   $$
   By assumption
   $$
   \frac2q+\frac np=1-s, \quad \mbox{with } p\in\left(\frac{n}{1-s},\infty\right]\mbox{and } s\in(0,1),
   $$
   it has
   $$
   s_1:=s+\frac np=1-\frac2q\in[s,1)\subseteq(0,1).
   $$
   Thus, by the aid of the assumption in Theorem \ref{thm1}, we always has $u\in L^q(0,T; B^{-s_1}_{\infty,\infty})$, for some $s_1\in(0,1)$. As a result, to prove Theorem \ref{thm1}, it suffice to prove the case that $p=\infty$. Because of this, without loss of generality, we suppose that $u\in L^q(0,T; B^{-s}_{\infty,\infty})$ for some $s\in(0,1)$ in the following proof.

   We first note that for any $s\in(0,1)$, one can choose the numbers $p_s\in(2+\frac2s, 2+\frac4s)$ and $q_s\in(2,2+\frac{2}{s+1})$, such that
  $$
  \frac{2}{p_s}+\frac{2}{q_s}=1.
  $$
  In fact, when $p$ and $q$ run in the intervals $(2+\frac2s, 2+\frac4s)$ and $(2,2+\frac{2}{s+1})$, respectively, then the quantity $\frac2p+\frac2q$ runs in the interval $(\frac{1+2s}{2+s}, 1+\frac{s}{s+1})$, and each point in this interval can be arrived at by $\frac2p+\frac2q$. The validity of the above equality then follows from the observation that $\frac{1+2s}{2+s}<1<1+\frac{s}{s+1}$ for any $s\in(0,1)$. By Lemma \ref{lemL1} and Lemma \ref{lemL3}, we have the following estimates
  \begin{eqnarray*}
    &&\|u\|_{p_s}\leq C \|u\|_{B^{-s}_{\infty,\infty}}^{1-\frac{2}{p_s}}\|u\|_{H^1} ^{\frac{4}{p_s}-s(1-\frac{2}{p_s})}\|u\|_{H^2}^{s(1-\frac{2}{p_s})- \frac{2}{p_s}},\\
    &&\|\nabla u\|_{q_s}\leq C  \|u\|_{B^{-s}_{\infty,\infty}}^{1-\frac{2}{q_s}}\|u\|_{H^1} ^{\frac{2}{q_s}-(1+s)(1-\frac{2}{q_s})}\|u\|_{H^2}^{(1+s)(1-\frac{2} {q_s})},
  \end{eqnarray*}
  and thus
  \begin{align*}
    \|u\|_{p_s}\|\nabla u\|_{q_s}\leq&C\|u\|_{B^{-s}_{\infty,\infty}}^{ 2-\frac{2}{p_s}-\frac{2}{q_s}}\|u\|_{H^1} ^{\frac{4}{p_s}-s(1-\frac{2}{p_s})+ \frac{2}{q_s}-(1+s)(1-\frac{2}{q_s})}\\
    &\times\|u\|_{H^2}^{s(1-\frac{2}{p_s})- \frac{2}{p_s}+(1+s)(1-\frac{2} {q_s})}=C\|u\|_{B^{-s}_{\infty,\infty}}\|u\|_{H^1}^{1-s} \|u\|_{H^2}^s,
  \end{align*}
  where $C$ is a positive constant.

  With the aid of the above estimate, by hypothesis (H4), it follows from the H\"older and Young inequalities that
  \begin{align*}
    \frac{d}{dt}\|\nabla u\|_2^2&+c_1\|\Delta u\|_2^2\leq c_1'\int_{\mathbb R^n}|u|^2|\nabla u|^2 dx\\
    \leq& c_1'\|u\|_{p_s}^{2}\|\nabla u\|_{q_s}^{2}
    \leq C\|u\|_{B^{-s}_{\infty,\infty}}^2\|u\|_{H^1}^{2(1-s)} \|u\|_{H^2}^{2s}\\
    \leq& C\|u\|_{B^{-s}_{\infty,\infty}}^2(\|u\|_{H^1}^{2(1-s)} \|\Delta u\|_{2}^{2s}+\|u\|_{H^1}^2)\\
    \leq&\frac{c_1}{2}\|\Delta u\|_2^2+C(\|u\|_{B^{-s}_{\infty,\infty}}^{\frac{2}{1-s}} +\|u\|_{B^{-s}_{\infty,\infty}})\|u\|_{H^1}^2\\
    \leq&\frac{c_1}{2}\|\Delta u\|_2^2+C(\|u\|_{B^{-s}_{\infty,\infty}}^{\frac{2}{1-s}} +1)\|u\|_{H^1}^2,
  \end{align*}
  and thus
  \begin{align*}
    \frac{d}{dt}\|\nabla u\|_2^2+\frac{c_1}{2}\|\Delta u\|_2^2\leq C(\|u\|_{B^{-s}_{\infty,\infty}}^{\frac{2}{1-s}} +1)(\|\nabla u\|_2^2+\|u\|_2^2),
  \end{align*}
  for any $t\in(0,T)$. Applying the Gronwall inequality to the above inequality and recalling the hypothesis (H3), we obtain
  \begin{align*}
    &\sup_{0\leq \tau\leq t}\|\nabla u\|_2^2+\frac{c_1}{2}\int_0^t\|\Delta u\|_2^2d\tau\\
    \leq& e^{C\int_0^t(1+\|u\|_{B^{-s}_{\infty,\infty}}^{\frac{2}{1-s}})d\tau}\left( \|\nabla u_0\|_2^2+\mathcal M(t)\int_0^t(1+\|u\|_{B^{-s}_{\infty,\infty}}^{\frac{2}{1-s}}) d\tau\right)\\
    \leq& e^{C\int_0^T(1+\|u\|_{B^{-s}_{\infty,\infty}}^{\frac{2}{1-s}}) d\tau}\left( \|\nabla u_0\|_2^2+\mathcal M(T)\int_0^T(1+\|u\|_{B^{-s}_{\infty,\infty}}^{\frac{2}{1-s}}) d\tau\right),
  \end{align*}
  for any $t\in(0,T)$.
  On account of this, and recalling the hypothesis (H3), we obtain the a priori estimate
  $$
  \sup_{0\leq t<T}\|u\|_{H^1}^2+\int_0^T\|u\|_{H^2}^2dt\leq \mathcal G_1<\infty,
  $$
  for any $t\in(0,T)$, where $\mathcal G_1$ is a positive constant.

  Choose a positive time $T_{**}\in(0,T_*)$. By the hypothesis (H1) and recalling the above a priori estimate, starting from the time $T_{**}$, there is a unique solution $u_{**}$ to $(\mathscr P)$ on $\mathbb R^n\times(T_{**}, T_{**}+\mathcal T)$, with initial data $u(T_{**})$, such that $u_{**}\in C([T_{**}, T_{**}+\mathcal T]; H^1)$, where $\mathcal T=\mathcal T(\mathcal G_1)$. Define
  $$
  \tilde u=\left\{
  \begin{array}{l}
  u(x,t),\quad t\in[0,T_{**}),\\
  u_{**}(x,t),\quad t\in[T_{**},T_{**}+\mathcal T(\mathcal G_1)],
  \end{array}
  \right.
  $$
  then, by hypothesis (H2), $\tilde u$ is a solution to $(\mathscr P)$ on $\mathbb R^n\times(0,T_{**}+\mathcal T(\mathcal G_1))$. Note that $\mathcal T(\mathcal G_1)$ is independent of $T_{**}$, therefore, one can choose $T_{**}$ close enough to $T_*$, such that $T_{**}+\mathcal T>T_*$, and as a result $\tilde u$ is an extension of $u$ to a time beyond $T_*$. This completes the proof.
\end{proof}

\begin{proof}[\textbf{Proof of Theorem \ref{thm2}}]
  Similar to the situation encountered in the proof of Theorem \ref{thm1}, it suffices to consider the case $p=\infty$. Besides, without loss of generality, we suppose that
  $$
  \sup_{0\leq t\leq T}\|\nabla u\|_2^2\geq1,
  $$
  otherwise, we then already have the a priori estimate $\sup_{0\leq t\leq T}\|\nabla u\|_2^2\leq1$, on account of which, by the same argument as that in the last paragraph of the proof of Theorem \ref{thm1}, one can easily prove the conclusion.

  We first consider the case that $s\in(-1,1)$. By Lemma \ref{lemL3}, it follows from the Young inequality that
  \begin{align}
    \|\nabla u\|_3^3\leq&C\|u\|_{B^s_{\infty,\infty}}\|\nabla u\|_2^{1+s}\|\nabla u\|_{H^1}^{1-s}\nonumber\\
    \leq&C\|u\|_{B^s_{\infty,\infty}}(\|\nabla u\|_2^2+\|\nabla u\|_2^{1+s}\|\Delta u\|_2^{1-s})\nonumber\\
    \leq&\frac{c_2}{2}\|\Delta u\|_2^2+C(\|u\|_{B^s_{\infty,\infty}}+
    \|u\|_{B^s_{\infty,\infty}}^{\frac{2}{1+s}})\|\nabla u\|_2^2\nonumber\\
    \leq&\frac{c_2}{2}\|\Delta u\|_2^2+C(1+
    \|u\|_{B^s_{\infty,\infty}}^{\frac{2}{1+s}})\|\nabla u\|_2^2,\label{ne1}
  \end{align}
  and thus it follows from the hypothesis (H4') that
  $$
  \frac{d}{dt}\|\nabla u\|_2^2+\frac{c_2}{2}\|\Delta u\|_2^2\leq C(1+\|u\|_{B^s_{\infty,\infty}}^{\frac{2}{1+s}})\|\nabla u\|_2^2,
  $$
  for any $t\in(0,T)$.
  Applying the Gronwall inequality to the above inequality yields
  \begin{align*}
    &\sup_{0\leq\tau \leq t}\|\nabla u\|_2^2+\frac{c_2}{2}\int_0^t\|\Delta u\|_2^2d\tau\\
    \leq&e^{C\int_0^t(1+\|u\|_{B^s_{\infty,\infty}}^{\frac{2}{1+s}})d\tau} \|\nabla u_0\|_2^2\leq e^{C\int_0^T(1+\|u\|_{B^s_{\infty,\infty}}^{\frac{2}{1+s}})d\tau} \|\nabla u_0\|_2^2,
  \end{align*}
  for any $t\in(0,T)$. With the aid of this a priori estimate, recalling the hypothesis (H3), we have the a priori estimate
  $$
  \sup_{0\leq\tau\leq t} \|u\|_{H^1}^2+\int_0^t\|u\|_{H^2}^2d\tau\leq\mathcal G_2<\infty,
  $$
  for any $t\in(0,T)$, where $\mathcal G_2$ is a positive constant. Thanks to this a priori estimate, the same argument as that for the proof of Theorem \ref{thm1} yields the conclusion for the case $s\in(-1,1)$.

Now we consider the case that $(p,s)=(\infty,1)$. Let $\varepsilon$ be a sufficiently small positive number which will be determined later. Since $u\in L^1(0,T_*; B^1_{\infty,\infty}(\mathbb R^d))$, by the absolutely continuity of the integrals, there is a positive constant $\delta>0$, such that
\begin{equation*}
\int_{T_*-\delta}^{T_*}\|u\|_{B^1_{\infty,\infty}}dt\leq\varepsilon,
\end{equation*}
from which, by Lemma \ref{lemL5}, we have
\begin{equation}
  \label{ne2}
  \int_{T_*-\delta}^t\|\nabla u\|_\infty ds\leq C\left[\varepsilon\log\left( \int_{T_*-\delta}^t\|\nabla\Delta u\|_2ds+e\right)+1\right],
\end{equation}
for any $t\in[T-\delta, T)$.
We set
\begin{align*}
f_1(t)=&\|\nabla u\|_2^2(t)+c_2\int_{T_*-\delta}^t \|\Delta u\|_2^2 ds,\\
f_2(t)=&\|\Delta u\|_2^2(t)+c_3\int_{T_*-\delta}^t\|\nabla\Delta u\|_2^2 ds,
\end{align*}
for any $t\in[T_*-\delta, T_*)$.

By the Ladyzhenskaya and Sobolev embedding inequalities, we deduce
\begin{align}
  \int_{T_*-\delta}^t\| \nabla u\|_4^4ds\leq&C\int_{T_*-\delta}^t \|\nabla u\|_2^2\|\Delta u\|_2^2ds\nonumber\\
  \leq&C\sup_{T_*-\delta\leq s\leq t}f_1^2(s),\quad t\in[T_*-\delta, T_*),\label{ne3-2}
\end{align}
for $d=2$, and
\begin{align}
  \int_{T_*-\delta}^t\|\nabla u\|_4^4ds\leq&C\int_{T_*-\delta}^t \|\nabla u\|_2\|\Delta u\|_2^3ds
  \leq C\int_{T-\delta}^t f_1^{\frac12}(s)f_2^{\frac12}(s)\|\Delta u\|_2^2ds\nonumber\\
  \leq&Cf_1(t)\sup_{T_*-\delta\leq s\leq t}f_1^{\frac12}(s)f_2^{\frac12}(s),\quad t\in[T_*-\delta, T_*), \label{ne3-3}
\end{align}
for $d=3$.
It follows from (H4') and (H5) that
\begin{eqnarray*}
  &&f_1'(t)\leq C \|\nabla u\|_\infty f_1(t), \\
  &&f_2'(t)\leq C \|\nabla u\|_\infty f_2(t)+C\| \nabla u\|_4^4,
\end{eqnarray*}
from which, by the Gronwall inequality, recalling (\ref{ne2})---(\ref{ne3-3}), and choosing $\varepsilon$ small enough, we have
\begin{align}
  f_1(t)\leq&C\left(\int_{T_*-\delta}^t\| \nabla\Delta u\|_2ds +1\right)^{C\varepsilon}f_1(T_*-\delta)\nonumber\\
  \leq& C(f_2(t)+1)^{C\varepsilon} \leq C(f_2(t)+1)^{\frac16},\label{ne4}
\end{align}
and
\begin{align}
  f_2(t)\leq&C\left(\int_{T_*-\delta}^t\| \nabla\Delta u\|_2ds +1\right)^{C\varepsilon} \left(f_2(T_*-\delta)+\int_{T_*-\delta}^t\| \nabla u\|_4^4ds\right)\nonumber\\
  \leq&
  \left\{
  \begin{array}{lr}
  C(f_2(t)+1)^{\frac16} \left(1+ \displaystyle\sup_{T_*-\delta\leq s\leq t}f_1^2(s)\right),&d=2,\\
  C(f_2(t)+1)^{\frac16} \left(1+ f_1(t)\displaystyle\sup_{T_*-\delta\leq s\leq t}f_1^{\frac12}(s)f_2^{\frac12}(s)\right),&d=3,
  \end{array}
  \right. \label{ne5}
\end{align}
for any $t\in[T_*-\delta, T_*)$.

Define
$$
F_1(t)=\sup_{T_*-\delta\leq s\leq t}f_1(s)+1, \quad F_2(t)=\sup_{T_*-\delta\leq s\leq t}f_2(s)+1,
$$
for any $t\in[T_*-\delta, T_*)$.
Then, it follows from (\ref{ne4}) that $F_1(t)\leq CF_2^{\frac16}(t)$, and thus it follows from (\ref{ne5}) that $F_2(t)\leq CF_2^{\frac12}(t)$, for $d=2$; similarly,
$F_2(t)\leq CF_2^{\frac16}(t)F_2^{\frac14}(t)F_2^{\frac12}(t)=CF_2^{\frac{11}{12}}(t),$
for $d=3$. Hence $F_1(t),F_2(t)\leq C<\infty,$
for any $t\in[T_*-\delta, T_*)$. On account of this uniform in time estimates, and recalling the hypothesis (H3), one can easily obtain the a priori estimate
  $$
  \sup_{\delta\leq \tau\leq t}\|u\|_{H^2}^2+\int_\delta^t\|u\|_{H^3}^2d\tau\leq\mathcal G_3<\infty,
  $$
  for any $t\in(T_*-\delta, T_*)$, where $\mathcal G_3$ is a positive constant. On account of this estimate, the same argument as before yields the conclusion.
\end{proof}


\section*{Acknowledgments}
{The authors would like to thank Professor Zhouping Xin for his valuable discussions and continuous encouragements. J. Li was was supported in part by the National Natural Science Foundation of
China (11971009 and 11871005), by the Key Project of National Natural Science
Foundation of China (12131010), and by the Guangdong Basic and Applied Basic
Research Foundation (2019A1515011621, 2020B1515310005, 2020B1515310002, and
2021A1515010247). M. Wang is partially supported by NSF of China under Grant No. 11771388.
W. Wang was supported by NSFC under grant 12071054, 11671067 and by Dalian High-level Talent Innovation Project (Grant 2020RD09).
}
\par

\end{document}